\documentclass[12pt,a4paper]{article}%
\usepackage[utf8]{inputenc}
\usepackage{hyperref}
\usepackage{amsmath}
\usepackage{amsfonts}
\usepackage{amssymb}
\usepackage{graphicx}%
\providecommand{\U}[1]{\protect\rule{.1in}{.1in}}
\newtheorem{theorem}{Theorem}

\newtheorem{conjecture}[theorem]{Conjecture}
\newtheorem{corollary}[theorem]{Corollary}

\newtheorem{lemma}[theorem]{Lemma}

\newtheorem{proposition}[theorem]{Proposition}
\newtheorem{remark}[theorem]{Observation}

\newenvironment{proof}[1][Proof]{\noindent\textbf{#1.} }{\ \hfill \rule{0.5em}{0.5em}\bigskip}
\graphicspath{{D:/Dropbox/Riste-Sedlar/MetricDIm/Jelena/Slike/}}

\begin{document}

\title{Mixed metric dimension of graphs with edge disjoint cycles}
\author{Jelena Sedlar$^{1}$,\\Riste \v Skrekovski$^{2,3}$ \\[0.3cm] {\small $^{1}$ \textit{University of Split, Faculty of civil
engineering, architecture and geodesy, Croatia}}\\[0.1cm] {\small $^{2}$ \textit{University of Ljubljana, FMF, 1000 Ljubljana,
Slovenia }}\\[0.1cm] {\small $^{3}$ \textit{Faculty of Information Studies, 8000 Novo
Mesto, Slovenia }}\\[0.1cm] }
\maketitle

\begin{abstract}
In a graph $G$, the cardinality of the smallest ordered set of vertices that
distinguishes every element of $V(G)\cup E(G)$ is called the mixed metric
dimension of $G$. In this paper we first establish the exact value of the
mixed metric dimension of a unicycic graph $G$ which is derived from the
structure of $G$. We further consider graphs $G$ with edge disjoint cycles in
which a unicyclic restriction $G_{i}$ is introduced for each cycle $C_{i}.$
Applying the result for unicyclic graph to each $G_{i}$ then yields the exact
value of the mixed metric dimension of such a graph $G$. The obtained formulas
for the exact value of the mixed metric dimension yield a simple sharp upper
bound on the mixed metric dimension, and we conclude the paper conjecturing
that the analogous bound holds for general graphs with prescribed cyclomatic number.

\end{abstract}



\section{Introduction}

The distance between vertices $u$ and $v$ from a graph $G$ will be denoted by
$d_{G}(u,v)$ (or simply $d(u,v)$ if no confusion arises). If there is a vertex
$s\in V(G)$ such that $d_{G}(u,s)\neq d_{G}(v,s)$, then we say that $s$
\emph{distinguishes} (or \emph{resolves}) $u$ and $v$. If any two vertices $u$
and $v$ from $G$ are distinguished by at least one vertex of a subset
$S\subseteq V(G)$, then we say that $S$ is a \emph{metric generator} for $G$.
The cardinality of the smallest metric generator is called the \emph{metric
dimension} of $G$, and it is denoted by $\dim(G)$. This notion for graphs was
independently introduced by \cite{Harary1976} and \cite{Slater1975}, under the
names resolving sets and locating sets, respectively. Even before this notion
was introduced for the realm of metric spaces \cite{Blumenthal1953}.

The concept of metric dimension was recently extended from resolving vertices
to resolving edges of a graph by Kelenc, Tratnik and Yero~\cite{Kel}.
Similarly as above, a vertex $s\in V(G)$ \emph{distinguishes} two edges
$e,f\in E(G)$ if $d_{G}(s,e)\neq d_{G}(s,f)$, where $d_{G}(e,s)=d_{G}%
(uv,s)=\min\{d(u,s),d(v,s)\}$. A set of vertices $S\subseteq V(G)$ is an
\emph{edge metric generator} for $G$, if any two edges of $G$ are
distinguished by a vertex of $S$. The cardinality of the smallest edge metric
generator is called the \emph{edge metric dimension} of $G$, and it is denoted
by $\mathrm{edim}(G)$. This variation of metric dimension also attracted
interest (see \cite{Peterin2020}, \cite{SedSkreBounds}, \cite{Zhu},
\cite{Zubrilina}).

Very recently the mixed metric dimension was introduced by Kelenc, Kuziak,
Taranenko, and Yero~\cite{Kelm}. Here $S$ is a \emph{mixed metric generator},
if it distinguishes any two elements from $V(G)\cup E(G)$, and similarly as
above, the size of the smallest such a set $S$ is the \emph{mixed metric
dimension} of $G$, and it is denoted by $\mathrm{mdim}(G)$. The paper contains
lower and upper bounds for various graph classes. For a wider and systematic
introduction of the topic metric dimension that encapsulates all three above
mentioned variations, we recommend the PhD thesis of Kelenc~\cite{KelPhD}.

In this paper, after this introductory section, we will first provide the
necessary definitions and notation in the next section. In the third section
the exact value of the mixed metric dimension for unicyclic graphs is
obtained, while in the fourth section that result is extended to graphs with
more cycles than one, but which cycles are edge disjoint. We conclude the
paper with the fifth section in which concluding remarks and possible
directions for further research are given.

\section{Preliminaries}

In a graph $G$, let $n$ denote the number of vertices and $m$ the number of
edges. The \emph{degree} of a vertex $v\in V(G)$ is defined as the number of
neighbors of $v$ in $G$ and denoted by $\deg(v).$ We say that a vertex $v\in
V(G)$ is a \emph{leaf} if $\deg(v)=1.$ By $L_{1}(G)$ we denote the number of
leaves in $G$.

A graph $G$ with edge disjoint cycles is called a \emph{cactus}. A cactus with
only one cycle is also called a \emph{unicyclic} graph. Let $G$ be a cactus,
let $C$ be a cycle in $G$ and let $v$ be a vertex from $C.$ By $T_{v}$ we will
denote the connected component of $G-E(C)$ containing the vertex $v.$ If $G$
is a unicyclic graph then $T_{v}$ is a tree for every $v\in V(C),$ but if $G$
has more than one cycle then $T_{v}$ need not be a tree, i.e. it may contain
cycles. We say that a vertex $v$ from a cycle $C$ is a \emph{root} vertex on
$C$ if $T_{v}$ is nontrivial, otherwise we say that $v$ is a \emph{non-root}
vertex. The number of all root vertices on a cycle $C$ is denoted by
$\mathrm{rt}(C).$ Obviously, every vertex on a cycle of the degree $\geq3$ is
a root vertex, while all vertices on a cycle which are of degree $2$ are
non-root vertices. For three vertices $u,$ $v$ and $w$ from a cycle $C$ in a
cactus $G$ we say that they form a \emph{geodesic triple} of vertices if
$d(u,v)+d(v,w)+d(w,u)=\left\vert V(C)\right\vert .$

The \emph{cyclomatic number} $c(G)$ of a graph $G$ is defined by $c(G)=m-n+1.$
Observe that for a cactus graph $G$ the cyclomatic number $c(G)$ equals the
number of cycles in $G.$ The \emph{(vertex) connectivity} $\kappa(G)$ of a
graph $\emph{G}$ is the minimum size of a vertex cut, i.e. any subset of
vertices $S\subseteq V(G)$ such that $G-S$ is disconnected or has only one
vertex. We say that a graph $G$ is $k$\emph{-connected} if $\kappa(G)\geq k.$
Notice that for a connected graph $G$ which has a leaf it obviously holds that
$\kappa(G)=1.$ Also, notice that for every cactus graph $G$ with at least two
cycles it holds that $\kappa(G)=1,$ even if $G$ does not contain a leaf.

Let $S\subseteq V(G)$ be a set of vertices in a cactus graph $G$. We say that
a vertex $v$ from a cycle $C$ in $G$ is $S$\emph{-active} if $T_{v}$ contains
a vertex from $S$. By $a_{S}(C)$ we denote the number of $S$-active vertices
on $C$. Any shortest path between two vertices from $S$ is called a
$S$\emph{-closed} path. Let $x$ and $x^{\prime}$ be a pair of elements from
the set $V(G)\cup E(G)$. We say that a pair $x$ and $x^{\prime}$ is
\emph{enclosed} by $S$ if there is a $S$-closed path containing $x$ and
$x^{\prime}$. We say that a pair $x$ and $x^{\prime}$ is \emph{half-enclosed}
by $S$ if there is a vertex $s\in S$ such that a shortest path from $s$ to $x$
contains $x^{\prime}$ or a shortest path from $s$ to $x^{\prime}$ contains $x$.

\begin{remark}
\label{Obs_enclosure}Let $G$ be a graph, let $S\subseteq V(G)$ be a set of
vertices in $G$ and let $x$ and $x^{\prime}$ be a pair of elements from the
set $V(G)\cup E(G)$. If $x$ and $x^{\prime}$ are enclosed by $S,$ then $x$ and
$x^{\prime}$ are distinguished by $S.$ If $x$ and $x^{\prime}$ are
half-enclosed by $S$ then $x$ and $x^{\prime}$ are distinguished by $S$ in all
cases except possibly when $x$ and $x^{\prime}$ are a pair consisting of a
vertex and an edge which are incident to each other.
\end{remark}

The following interesting result was proved in~\cite{Kelm}.

\begin{proposition}
\label{Prop_Kelm}For every tree $T$, it holds $\mathrm{mdim}(T)=L_{1}(T)$.
\end{proposition}

In what follows we will give a counterpart of this result for unicyclic graphs
and afterwards extend that result to cactus graphs.


\section{Mixed dimension of unicyclic graphs}

We first want to give a characterization of the sets $S\subseteq V(G)$ which
are mixed metric generators in a unicyclic graph $G$.

\begin{lemma}
\label{Lemma_mixed} Let $G$ be a unicyclic graph with $C$ being its only
cycle. A set $S\subseteq V(G)$ is a mixed metric generator if and only if $S$
contains all leaves from $G$ and there is a geodesic triple of $S$-active
vertices on $C.$
\end{lemma}

\begin{proof}
Assume $S\subseteq V(G)$ is a mixed metric generator. We want to prove that
$S$ contains all leaves and that there is a geodesic triple of $S$-active
vertices on $C$. If this does not hold, then we have the following two possibilities.

First, if there is a vertex $v\in V(G)$ such that $\deg(v)=1$ and
$v\not \in S$, then let $e$ be the only edge from $E(G)$ incident to $v$ and
let $w$ be the other end-vertex of $e$. Then $w$ and $e$ are not distinguished
by $S$ which is a contradiction.

Second, if there is no geodesic triple of $S$-active vertices on $C$, then we
consider several cases with respect to the value of $a_{S}(C)$. If
$a_{S}(C)=0$, then $S=\phi$, so obviously $S$ cannot be a mixed metric
generator which is a contradiction. If $a_{S}(C)=1$, let $u$ be the only
$S$-active vertex on $C$ and let $v$ and $w$ be two neighbors of $u$ on $C.$
But then $v$ and $w$ are not distinguished by $S$. Finally, if $a_{S}(C)\geq2$
then let $u$ and $v$ be the two $S$-active vertices on $C$ on the greatest
possible distance. As there is no geodesic triple of $S$-active vertices on
$C,$ there must exist a neighbor $w$ of $v$ on the cycle $C$ such that
$d(u,v)\leq d(u,w)$. Notice that in this case $v$ and $vw$ are not
distinguished by $S,$ which is again a contradiction with $S$ being a mixed
metric generator.

Now we prove the other direction. So, assume that a set $S\subseteq V(G)$
contains all leaves and there is a geodesic triple of $S$-active vertices on
$C$. We want to prove $S$ is a mixed metric generator. Let $x$ and $x^{\prime
}$ be two elements from the set $V(G)\cup E(G).$ If $x$ and $x^{\prime}$ are
enclosed by $S$, Observation \ref{Obs_enclosure} implies they are
distinguished by $S$ and the proof is over. Assume therefore that $x$ and
$x^{\prime}$ are not enclosed by $S$. We distinguish the following two cases.

\medskip\noindent\textbf{Case 1:} $x$\emph{ or }$x^{\prime}$\emph{ does not
belong to the cycle }$C$\emph{. }Without loss of generality we may assume $x$
does not belong to $C.$ Let $v$ be the vertex on $C$ such that $x$ belongs to
$T_{v}.$ Notice that $x^{\prime}$ can neither belong to $T_{v}$ nor be an edge
incident to $v$ on $C,$ as in that case the existence of a geodesic triple of
$S$-active vertices on $C$ would imply $x$ and $x^{\prime}$ are enclosed by
$S$. Therefore, $x$ and $x^{\prime}$ are not incident to each other. Note that
there is a leaf $s$ in $T_{v}$ such that the shortest path from $s$ to
$x^{\prime}$ contains $x.$ Since $s$ as a leaf belongs to $S,$ this implies
$x$ and $x^{\prime}$ are half-enclosed by $S$ and then the claim follows from
Observation \ref{Obs_enclosure}.

\medskip\noindent\textbf{Case 2:} \emph{both }$x$\emph{ and }$x^{\prime}%
$\emph{ belong to the cycle }$C$\emph{.} Notice that $x$ and $x^{\prime}$
cannot be a vertex and an edge incident to each other, since in that case they
would certainly be enclosed by $S.$ Let $u$ and $v$ be two $S$-active vertices
on $C.$ The pair $x$ and $x^{\prime}$ is not distinguished by $u$ and $v$ only
if $u$ and $v$ are a pair of antipodal vertices. But then the pair $x$ and
$x^{\prime}$ is certainly distinguished by a third $S$-active vertex $w$ on
$C,$ which must exist as there is a geodesic triple of $S$-active vertices on
$C$. If $w$ is not contained in $S,$ the fact that $w$ distinguishes $x$ and
$x^{\prime}$ implies that a vertex $s\in S\cap T_{w}$ distinguishes them also,
and such $s$ must exist as $w$ is $S$-active. Therefore, the pair $x$ and
$x^{\prime}$ is distinguished by $S$ which concludes the proof.
\end{proof}

Recall that $\mathrm{rt}(C)$ denotes the number of root vertices on the cycle
$C$ in a unicyclic graph $G$. The above lemma now enables us to prove the
following theorem which gives the exact value of the mixed metric dimension in
a unicyclic graph.

\begin{theorem}
Let $G$ be a unicyclic graph with $C$ being its only cycle. Then
\[
\mathrm{mdim}(G)=L_{1}(G)+\max\{3-\mathrm{rt}(C),0\}+\Delta,
\]
where $\Delta=1$ if $\mathrm{rt}(C)\geq3$ and there is no geodesic triple of
root vertices on $C$, while $\Delta=0$ otherwise.
\end{theorem}

\begin{proof}
Let $S$ be the set of all leaves in $G.$ Note that the number of $S$-active
vertices on $C$ equals the number of root vertices, i.e. $a_{S}(C)=\mathrm{rt}%
(C).$ Now consider the following: if $\mathrm{rt}(C)<3$ then $\max
\{3-\mathrm{rt}(C),0\}>0$ vertices can certainly be added to $S$ obtaining
thus the set $S^{\prime},$ so that there is a geodesic triple of $S^{\prime}%
$-active vertices. Applying Lemma \ref{Lemma_mixed} to the set $S^{\prime}$
yields the desired result.

On the other hand, if $\mathrm{rt}(C)\geq3$ then there may or may not be a
geodesic triple of root vertices on $C.$ If there is a geodesic triple of root
vertices on $C,$ then the claim follows from applying Lemma \ref{Lemma_mixed}
to the set $S$. Yet, if $\mathrm{rt}(C)\geq3$ and there is no geodesic triple
of root vertices on $C,$ then $S$ is not a mixed metric generator, but then a
vertex from $C$ can certainly be added to $S$ so that it forms a geodesic
triple with any two root vertices on $C.$ Let $S^{\prime}$ be the set obtained
from $S$ by adding such a vertex to it, then The Lemma \ref{Lemma_mixed}
implies $S^{\prime}$ is a mixed metric generator, which concludes the proof.
\end{proof}

This theorem immediately yields the following simple upper bound on the mixed
metric dimension of a unicyclic graph.

\begin{corollary}
\label{Cor_UnicyclicBound}Let $G\not =C_{n}$ be a unicyclic graph. Then
$\mathrm{mdim}(G)\leq L_{1}(G)+2$ and the equality is attained precisely when
there is only one root vertex on the cycle of $G.$
\end{corollary}

\section{Mixed dimension of graphs with edge disjoint cycles}

We wish to extend the results for unicyclic graphs from previous section to
graphs with more than one cycle. The extension can be done quite naturally for
graphs in which cycles are edge disjoint, i.e. for cactus graphs. For every
cycle $C_{i}$ in a cactus graph $G$ there is a unicyclic subgraph $G_{i}$ of
$G$ in which $C_{i}$ is the only cycle. Moreover, such subgraphs cover the
whole cactus graph $G$. Applying the results from the previous section to
those subgraphs will yield the result. But let us introduce all the necessary
definitions more formally.

A standard notion is that for any path $P$ in a graph $G$ connecting vertices
$u$ and $v$ we say that $u$ and $v$ are its \emph{end-vertices}, while all
other vertices in $P$ are called \emph{internal} vertices. Let $G$ be a cactus
graph, let $C_{i}$ be a cycle in $G$ and let $u$ be a vertex in $G.$ We say
that a vertex $u$ \emph{gravitates} to the cycle $C_{i}$ if it belongs to
$C_{i}$ or the shortest path from $u$ to $C_{i}$ does not share its internal
vertices with any cycle in $G$ (see Figure \ref{Figure20}). A \emph{unicyclic
region} of a cycle $C_{i}$ is the subgraph $G_{i}$ of $G$ induced by all
vertices that gravitate to $C_{i}.$ Notice that unicyclic regions of two
distinct cycles $C_{i}$ and $C_{j}$ are not necessarily vertex disjoint, but
unicyclic regions of all cycles in $G$ do cover the whole $G.$ We say that a
vertex $v\in V(G_{i})$ is a \emph{boundary vertex} of a unicyclic region
$G_{i}$ if $v\in V(C_{j})$ for $j\not =i$ (see again the Figure \ref{Figure20}%
).\begin{figure}[h]
\begin{center}
\includegraphics[scale=1.0]{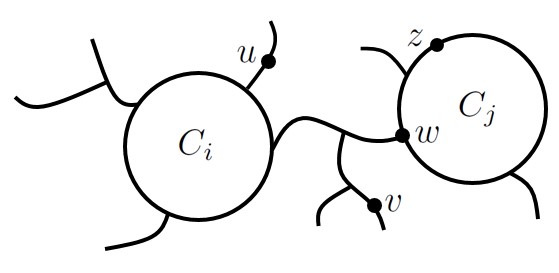}
\end{center}
\par
\caption{An illustration of a vertex gravitation: $u$ gravitates only to
$C_{i},$ $v$ gravitates to both $C_{i}$ and $C_{j},$ as does the vertex $w$
which is also a boundary vertex of $G_{i},$ and finally $z$ gravitates only to
$C_{j}$.}%
\label{Figure20}%
\end{figure}

As we will apply the results from the previous section to unicyclic regions of
cycles in a cactus graph $G$, for a set of vertices $S\subseteq V(G)$ we need
to introduce a corresponding set $S_{i}\subseteq V(G_{i}),$ so that if $S$ is
a mixed metric generator in $G$ then $S_{i}$ also is one in $G_{i}$. Note that
we cannot simply take $S_{i}=S\cap V(G_{i})$ since a pair of edges/vertices
from $G_{i}$ may in the graph $G$ be distinguished by a vertex $s\in S$ which
is outside of $G_{i}$. Therefore, for a set $S\subseteq V(G)$ we define a set
$S_{i}\subseteq V(G_{i})$ as the set obtained from $S\cap V(G_{i})$ by adding
all boundary vertices from $G_{i}$ to it. Note that the following holds: if
$S$ contains all leaves from $G$ then $S_{i}$ contains all leaves from $G_{i}%
$, also if there is a geodesic triple of $S$-active vertices on every cycle in
$G$ then there is a geodesic triple of $S_{i}$-active vertices on $C_{i}$ in
$G_{i}$.

\begin{lemma}
Let $G$ be a cactus graph with $c$ cycles $C_{1},\ldots,C_{c}$. Then
\[
\mathrm{mdim}(G)\leq L_{1}(G)+\sum_{i=1}^{c}\max\{3-\mathrm{rt}(C_{i}%
),0\}+\Delta,
\]
where $\Delta$ is the number of cycles $C_{i}$ for which $\mathrm{rt}%
(C_{i})\geq3$ and there is no geodesic triple of root vertices on the cycle
$C_{i}$.
\end{lemma}

\begin{proof}
We will prove the lemma by constructing a mixed metric generator $S$ of the
desired cardinality. For that purpose, let $S_{a}$ be the set of all leaves in
$G$ and note that $\left\vert S_{a}\right\vert =L_{1}(G).$ Let $S_{b}$ consist
of $\max\{3-\mathrm{rt}(C_{i}),0\}$ non-root vertices from every cycle $C_{i}%
$, which are chosen so that there is a geodesic triple of $(S_{a}\cup S_{b}%
)$-active vertices on the cycle $C_{i}$ if that is possible (it will be
possible for every cycle $C_{i}$ in which $\mathrm{rt}(C_{i})\leq2$, since for
any two vertices on a cycle it is possible to add a third one so that it forms
a geodesic triple with those two). Note that $\left\vert S_{b}\right\vert
=\sum_{i=1}^{c}\max\{3-\mathrm{rt}(C_{i}),0\}.$ Finally, consider all cycles
$C_{i}$ in $G$ with $\mathrm{rt}(C_{i})\geq3$ for which there is no geodesic
triple of $(S_{a}\cup S_{b})$-active vertices and let $S_{c}$ be composed of a
vertex from every such cycle chosen so that it forms a geodesic triple with
two $(S_{a}\cup S_{b})$-active vertices on that cycle. Note that $\left\vert
S_{c}\right\vert =\Delta.$ Now we define $S=S_{a}\cup S_{b}\cup S_{c}.$ Since
$S_{a},$ $S_{b}$ and $S_{c}$ are pairwise disjoint, we immediately obtain
\[
\left\vert S\right\vert =L_{1}(G)+\sum_{i=1}^{c}\max\{3-\mathrm{rt}%
(C_{i}),0\}+\Delta.
\]
It remains to prove that $S$ is a mixed metric generator in $G.$

Note that $S$ contains all leaves from $G$ and there is a geodesic triple of
$S$-active vertices on every cycle $C_{i}$ in $G.$ Let $x$ and $x^{\prime}$ be
a pair of elements from the set $V(G)\cup E(G).$ We distinguish the following
two cases.

\medskip\noindent\textbf{Case 1:} $x$\emph{ and }$x^{\prime}$\emph{ belong to
a same unicyclic region. }Assume $x$ and $x^{\prime}$ belong to the unicyclic
region $G_{i}$ of a cycle $C_{i}$. Note that the set $S_{i}$ contains all
leaves of the graph $G_{i}$ and there is a geodesic triple of $S_{i}$-active
vertices on the only cycle $C_{i}$ of $G_{i}.$ Therefore, according to Lemma
\ref{Lemma_mixed}, the set $S_{i}$ is a mixed metric generator in $G_{i}$
which means it distinguishes $x$ and $x^{\prime}.$ If $x$ and $x^{\prime}$ are
distinguished in $G_{i}$ by $s\in S_{i}\cap S,$ then they are distinguished by
the same $s$ in $G.$ If, on the other hand they are distinguished by $s\in
S\backslash S_{i},$ that implies $s$ is a boundary vertex of the region
$G_{i},$ but then the existance of a geodesic triple of $S$-active vertices on
every cycle in $G$ implies there is a vertex $s^{\prime}\in S$ outside $G_{i}$
such that
\[
d(x,s^{\prime})=d(x,s)+d(s,s^{\prime})\not =d(x^{\prime},s)+d(s,s^{\prime
})=d(x^{\prime},s^{\prime}),
\]
which means $s^{\prime}$ distinguishes $x$ and $x^{\prime}$ in $G.$ Therefore,
$x$ and $x^{\prime}$ are distinguished by $S$ and the case is proven.

\medskip\noindent\textbf{Case 2: }$x$\emph{ and }$x^{\prime}$\emph{ do not
belong to a same unicyclic region.} Assume $x$ belongs to the region $G_{i}$
and $x^{\prime}$ belongs to the region $G_{k},$ where $i\not =k$. We may
assume that $x^{\prime}$ does not belong to $G_{i}$ nor $x$ to $G_{k},$
otherwise this case would reduce to the previous case.

Let $v_{j}\in V(C_{j}),$ where $j\not =i$, be the boundary vertex of the
region $G_{i}$ closest to $x^{\prime}$. If $x$ and $x^{\prime}$ are incident
to each other, then $k=j$ and both $x$ and $x^{\prime}$ are incident or equal
to $v_{j}.$ Since both $C_{i}$ and $C_{j}$ contain a geodesic triple of
$S$-active vertices it immediately follows that $x$ and $x^{\prime}$ are
enclosed by $S$ and are therefore distinguished by $S$ according to
Observation \ref{Obs_enclosure}. Assume, therefore, that $x$ and $x^{\prime}$
are not incident to each other. If $x$ and $x^{\prime}$ are half-enclosed by
$S,$ then the fact they are not incident to each other together with
Observation \ref{Obs_enclosure} would further imply that they are
distinguished by $S.$ So, let us assume that $x$ and $x^{\prime}$ are neither
incident to each other nor half-enclosed by $S$.

Let $v$ be a vertex from the cycle $C_{i}$ such that $T_{v}$ is non-trivial.
If $x$ belongs to $T_{v}\cap G_{i}$ then the pair $x$ and $x^{\prime}$ is
certainly half-enclosed by a vertex $s\in S$ contained in $T_{v}$. The similar
argument holds when $x^{\prime}$ belongs to $T_{w}\cap G_{k}$ where $T_{w}$ is
a non-trivial connected component of a vertex $w$ from $C_{k}.$ As we assumed
that $x$ and $x^{\prime}$ are not half-enclosed by $S$, we can conclude that
$x$ belongs to $C_{i}$, $x^{\prime}$ belongs to $C_{j}$ and neither of them is
a root vertex on the corresponding cycle, i.e. if they are of degree $2$.

Now, let $u$ and $v$ be the vertices from the cycle $C_{i}$ and $C_{k}$
respectively, such that the distance between $u$ and $v$ is the smallest
possible. The fact that there is no vertex $s\in S$ half-enclosing the pair
$x$ and $x^{\prime},$ together with the fact that $C_{i}$ contains a geodesic
triple of $S$-active vertices, implies that a shortest path from $x$ to $u$
must contain a vertex distinct from $u$ which is $S$-active, denote it by $w$
(see Figure \ref{Figure22}). For the similar reason a shortest path from
$x^{\prime}$ to $v$ must contain a $S$-active vertex $z$ distinct from $v$.
Note that there is a shortest path from $x$ to $x^{\prime}$ that contains all
of the vertices $u,$ $v$, $w$ and $z.$ Let $s_{w}$ and $s_{z}$ be vertices
from $S$ contained in $T_{v}$ and $T_{z}$ respectively. Note the following: if
$x$ and $x^{\prime}$ are not distinguished by $s_{w}$ then
$d(x,w)=d(w,x^{\prime})$, but then $d(x,z)>d(z,x^{\prime})$ which further
implies $d(x,s_{z})>d(s_{z},x^{\prime}).$ Therefore, $x$ and $x^{\prime}$ are
distinguished by $S,$ so $S$ is a mixed metric generator. \begin{figure}[h]
\begin{center}
\includegraphics[scale=1.0]{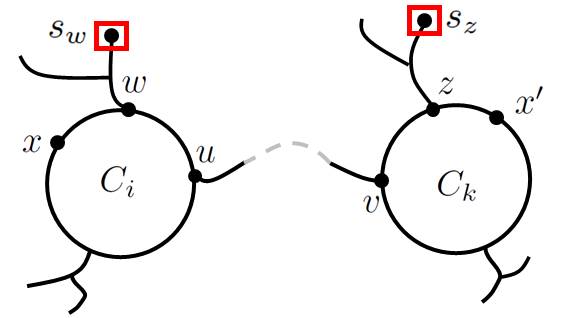}
\end{center}
\par
\caption{An illustration of a pair $x$ and $x^{\prime}$ which is neither
incident to each other nor half-enclosed by $S$: the shortest path from $x$ to
$u$ contains $S$-active vertex $w$, the shortest path from $x^{\prime}$ to $v$
contains $S$-active vertex $z$. Consequently, the pair $x$ and $x^{\prime}$ is
distinguished either by $s_{w}$ or $s_{z}$. }%
\label{Figure22}%
\end{figure}
\end{proof}

Now we show that the opposite inequality holds.

\begin{lemma}
Let $G$ be a cactus graph with $c$ cycles $C_{1},\ldots,C_{c}$. Then
\[
\mathrm{mdim}(G)\geq L_{1}(G)+\sum_{i=1}^{c}\max\{3-\mathrm{rt}(C_{i}%
),0\}+\Delta,
\]
where $\Delta$ is the number of cycles $C_{i}$ in $G$ for which $\mathrm{rt}%
(C_{i})\geq3$ and there is no geodesic triple of root vertices on the cycle
$C_{i}$.
\end{lemma}

\begin{proof}
Let $S$ be a set of vertices in $G$ such that
\[
\left\vert S\right\vert <L_{1}(G)+\sum_{i=1}^{c}\max\{3-\mathrm{rt}%
(C_{i}),0\}+\Delta.
\]
Then at least one of the following holds: 1) there is a leaf $v$ in $G$ not
contained in $S$, 2) there is a cycle $C_{i}$ in $G$ such that less than $3$
vertices from $C_{i}$ are $S$-active, 3) there is a cycle $C_{i}$ in $G$ such
that there is no geodesic triple of $S$-active vertices on $C_{i}.$ Note that
the second and the third case reduce to the existence of a cycle $C_{i}$
without geodesic triple of $S$-active vertices.

It is sufficient to prove that $S$ is not a mixed metric generator. In the
first case, let $e$ be the edge incident to a leaf $v$ not contained in $S$
and let $w$ be the end-vertex of the edge $e$ distinct from $v.$ Then $e$ and
$w$ are not distinguished by $S,$ so $S$ is not a mixed metric generator. In
both second and third case, let $C_{i}$ be a cycle without a geodesic triple
of $S$-active vertices. If there are no $S$-active vertices on $C,$ then
$S=\phi$, so $S$ cannot be a mixed metric generator. If there is only one
$S$-active vertex $u$ on $C_{i},$ then neighbors $v$ and $w$ of $u$ on $C_{i}$
are not distinguished by $S.$ Finally, if there are at least two $S$-active
vertices on $C_{i},$ let $u$ and $v$ be two $S$-active vertices on $C_{i}$
such that $d(u,v)$ is maximum possible and let $w$ be the neighbor of $v$ on
$C_{i}$ such that $d(u,v)\leq d(u,w).$ Notice that such a $w$ must exist, as
there is no geodesic triple of $S$-active vertices on $C$. Then $v$ and $vw$
are not distinguished by $S,$ so $S$ is not a mixed metric generator.
\end{proof}

The previous two lemmas immediately yield the following result.

\begin{theorem}
\label{Tm_cactus}Let $G$ be a cactus graph with $c$ cycles $C_{1},\ldots
,C_{c}$. Then
\[
\mathrm{mdim}(G)=L_{1}(G)+\sum_{i=1}^{c}\max\{3-\mathrm{rt}(C_{i}%
),0\}+\Delta,
\]
where $\Delta$ is the number of cycles $C_{i}$ in $G$ for which $\mathrm{rt}%
(C_{i})\geq3$ and there is not a geodesic triple of root vertices on the cycle
$C_{i}$.
\end{theorem}

Theorem \ref{Tm_cactus} enables us to easily draw from it a simple upper bound
for the mixed metric dimension of a cactus graph, but we first need to
introduce several notions and observations. First, we say that cycles $C_{i}$
and $C_{j}$ of a cactus graph $G$ are \emph{neighbors} if there is a path from
$C_{i}$ to $C_{j}$ which does not share an internal vertex with any cycle in
$G.$ Since we assumed $G$ is connected, if $G$ has more than one cycle then
every cycle in $G$ has at least one neighboring cycle. Therefore,
$\mathrm{rt}(C_{i})\geq1$ for every cycle $C_{i}$ even without leaves present
in $G.$

\begin{corollary}
\label{Cor_leafless}Let $G$ be a cactus graph with $c\geq2$ cycles and without
leaves. Then $\mathrm{mdim}(G)\leq2c$ and the equality is attained if and only
if every cycle in $G$ has exactly one root vertex.
\end{corollary}

\begin{proof}
Considering the formula for $\mathrm{mdim}(G)$ from Theorem \ref{Tm_cactus},
note that $L_{1}(G)=0$ since $G$ has no leaves. Also, note that every cycle
$C_{i}$ contributes to $\mathrm{mdim}(G)$ by either $2$ or $1.$ The
contribution of the cycle $C_{i}$ to $\mathrm{mdim}(G)$ equals $2$ when
$\mathrm{rt}(G)=1$ (since then $\max\{3-\mathrm{rt}(C_{i}),0\}=2$ and $C_{i}$
contributes $0$ to $\Delta$). On the other hand, the contribution of $C_{i}$
to $\mathrm{mdim}(G)$ can be $1$ in two situations: when $\mathrm{rt}(G)=2$
(since then $\max\{3-\mathrm{rt}(C_{i}),0\}=1$ and $C_{i}$ contributes $0$ to
$\Delta$) or when $\mathrm{rt}(C_{i})\geq3$ and there is no geodesic triple of
root vertices on $C_{i}$ (since then $\max\{3-\mathrm{rt}(C_{i}),0\}=0$ and
$C_{i}$ contributes $1$ to $\Delta$).

We conclude that $\mathrm{mdim}(G)$ will be the greatest when all cycles in
$G$ have $\mathrm{rt}(G)=1$ in which case $\mathrm{mdim}(G)=c\cdot
\max\{3-\mathrm{rt}(C_{i}),0\}=2c$ (an example of such a graph is illustrated
by Figure \ref{Figure23}). For any cactus graph the value of $\mathrm{mdim}%
(G)$ can be only smaller, i.e. $\mathrm{mdim}(G)\leq2c$, which concludes the proof.
\end{proof}

\begin{figure}[h]
\begin{center}
\includegraphics[scale=1.0]{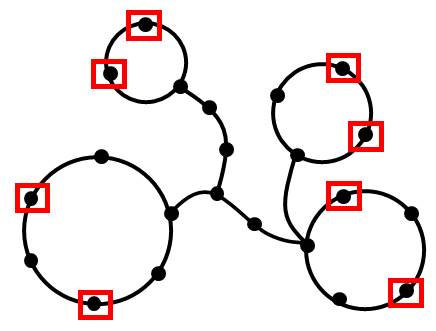}
\end{center}
\par
\caption{An example of a graph with $\mathrm{mdim}(G)=2c$. A mixed metric
generator $S$ of the smallest cardinality is emphasized on the graph, it
consists of two vertices from every cycle chosen so that there is a geodesic
triple of $S$-active vertices on every cycle.}%
\label{Figure23}%
\end{figure}

Notice that the presence of leafs in a cactus graph with at least two cycles
does not decrease the minimum number of root vertices on every cycle, i.e.
even if a cactus graph has a leaf it still holds $\mathrm{rt}(C_{i})\geq1$ for
every cycle $C_{i}$. Also, adding leaves to a leafless cactus graph with at
least two cycles does not necessarily increase $\mathrm{rt}(C_{i})$ either,
since leaves can be added to vertices on cycles which are already root
vertices or to the internal vertices of the paths connecting cycles.
Therefore, the following corollary also holds.

\begin{corollary}
\label{Cor_cactusBound}Let $G$ be a cactus graph with $c\geq2$ cycles. Then
$\mathrm{mdim}(G)\leq L_{1}(G)+2c$ and the bound is obtained if and only if
every cycle in $G$ has exactly one root vertex.
\end{corollary}

\section{Concluding remarks}

The natural folowing step would be to consider graphs in which cycles are not
necessarily edge disjoint, i.e. general graphs. Before we proceede with a bit
more detailed comment on that possible direction of the further research, let
us introduce a result which is of interest in such considerations.

\begin{proposition}
\label{Prop_kapa3}Let $G$ be a $3$-connected graph. Then $\mathrm{mdim}%
(G)<2c(G).$
\end{proposition}

\begin{proof}
Notice that the definition of $\mathrm{mdim}(G)$ implies $\mathrm{mdim}(G)\leq
n.$ For a graph $G$ with $\kappa(G)\geq3$ it holds that $m\geq\left\lceil
\frac{3}{2}n\right\rceil ,$ i.e. $m>\frac{3}{2}n-1$ which is equivalent to
$n<2m-2n+2.$ Therefore, for such a graph we have
\[
\mathrm{mdim}(G)\leq n<2m-2n+2=2c(G)
\]
and we are done.
\end{proof}

Let us now summarize what has been done so far in order to comment what might
be done next. So far we have proven that in the graphs with edge disjoint
cycles, i.e. cactus graphs, it holds that $\mathrm{mdim}(G)\leq L_{1}(G)+2c$
where $c$ is the number of cycles in $G.$ Notice that this bound was proved
for trees and unicyclic graphs, as special subclasses of cactus graphs, in
Proposition \ref{Prop_Kelm} and Corollary \ref{Cor_UnicyclicBound}
respectively. Finally, the bound was proven in Corollary \ref{Cor_cactusBound}
for general cactuses. Recalling that in graphs with edge disjoint cycles the
number of cycles $c$ equals the cyclomatic number $c(G)$, this bound can be
generalized as $\mathrm{mdim}(G)\leq L_{1}(G)+2c(G)$ and the question posed
whether that bound holds for general graphs. We conjecture that the answer to
that question is affirmative.

\begin{conjecture}
\label{Con1}Let $G\not =C_{n}$ be a graph and $c(G)$ its cyclomatic number.
Then
\[
\mathrm{mdim}(G)\leq L_{1}(G)+2c(G).
\]

\end{conjecture}

Notice that Proposition \ref{Prop_kapa3} implies that Conjecture \ref{Con1}
holds for all $3$-connected graphs, as in such graphs there are no leaves.
Therefore, it remains to verify that the conjecture holds for other
$1$-connected graphs besides cactus graphs and for $2$-connected graphs.

\bigskip\noindent\textbf{Acknowledgements.}~~The authors acknowledge partial
support Slovenian research agency ARRS program \ P1--0383 and ARRS project
J1-1692 and also Project KK.01.1.1.02.0027, a project co-financed by the
Croatian Government and the European Union through the European Regional
Development Fund - the Competitiveness and Cohesion Operational Programme.

\end{document}